\documentclass[english,11pt]{article}

\usepackage{amsmath,amsbsy,amscd,amssymb,graphicx,epsfig,color}
\usepackage{comment}
\usepackage{todonotes}  %%para agregar notas
\usepackage{enumerate}
\usepackage{pdflscape}
\usepackage{authblk}
\usepackage{amsthm}

\newtheorem{thm}{Theorem}
\newtheorem{lem}[thm]{Lemma}

\title{The Demand Adjustment Problem via Inexact Restoration Method}

\author[2,4]{Jorgelina Walpen}
  \author[2]{Elina M. Mancinelli}
\author[1]{Pablo A. Lotito}
\author[2,3]{Lisandro Parente}

\affil[1]{CONICET - Pladema, Universidad Nacional del Centro de la Provincia de Buenos Aires, Campus Universitario Paraje Arroyo Seco, (B7000) Tandil, Argentina. plotito@exa.unicen.edu.ar}
 \affil[2]{FCEIA, Universidad Nacional de Rosario, Pellegrini
250, (2000) Rosario, Argentina, elina@fceia.unr.edu.ar,
walpen@fceia.unr.edu.ar.}
 \affil[3]{CONICET - CIFASIS, Ocampo y Esmeralda, (S2000EZP) Rosario, Argentina. parente@cifasis-conicet.gov.ar}
\affil[4]{Corresponding author: Jorgelina Walpen - walpen@fceia.unr.edu.ar}
\date{}                     %% if you don't need date to appear
\providecommand{\keywords}[1]{\textbf{\textit{Keywords:}} #1}

\begin{document}

%%%%%%%%%%%%%%%%%%%%%%%%%%%%%%%%%%%%%%%%%%%%%
%\begin{frontmatter}

\maketitle

\begin{abstract}
In this work, the Demand Adjustment Problem (DAP) associated to urban
traffic planning is studied. The framework for the formulation of the DAP is
mathematical programming with equilibrium constraints. In
particular, if the optimization program associated to
the equilibrium constraint is considered, the DAP results in a bilevel optimization
problem. In this approach the DAP via the Inexact Restoration method is treated.

\end{abstract}


%\end{frontmatter}

%%===============================
\section{Introduction}
%%===============================
The Demand Adjustment Problem (DAP) consists in the estimation of the origin-destination matrix
(OD matrix) of a congested transport network. This problem is of
remarkable importance in the transportation planning process. This
matrix stores the number of trips originating and terminating in
each origin-destination pair.

The problem of adjusting the OD matrix can be modeled as an
optimization problem with equilibrium constraints and reformulated
as a bilevel problem. Among its drawbacks it has bad mathematical
properties which make it difficult to solve it. Some of them are:
non convexity, non differentiability, huge dimensions of real size
problems and the fact that the point-set-mapping which gives the
equilibrium flows is not explicitly known.

This version of the problem has been treated by many authors, some
whose works are: Nguyen S. (1977,\cite{Nguyen}), Spiess
(1990,\cite{Spiess90}), Chen and Florian (1996,\cite{C-F}), Yang et
al. (1992,\cite{Yangetal}), Codina and Barcel\'o (2004,\cite{C-B}),
Codina and Montero (2006,\cite{C-M}), Lundgren and Peterson
(2008,\cite{L-P}), Lotito and Parente (2015,\cite{LP:14}) and Walpen, Mancinelli and Lotito (2015,
\cite{WML}).

It is of remarkable importance that, with the only exception of \cite{LP:14}, all the methods
proposed in the mentioned papers are heuristics. In general, no
convergence proofs are given due to the fact that no appropriate
characterization of optimality points is available.
These methods have another characteristic in common: they all
generate sequences of feasible points through their iterations and
test the descent of the objective function of the problem.

In \cite{LP:14}, instead, the DAP is formulated as a general mathematical program with complementarity constraints (MPCC). Applying a lifting method (see \cite{Ste:12},\cite{IPS:12}), a necessary optimality condition is obtained in terms of a large non-linear semismooth system which is solved with a Newton-type method. However, the price to pay is the increase of the numerical problem size and the fact that the lower level structure is missed.

This work puts towards an approach to treat the DAP via Inexact Restoration. This method, originally proposed by
Martinez in \cite{Martinez98} and \cite{Martinez01} to solve
optimization problems with no linear constraints, has been adapted
to solve bilevel problems by Andreani et. al. in \cite{Andreani}.

The Inexact Restoration Method deals separately with feasibility and
optimality at each iteration. In the feasibility  stage, called
restoration phase, it seeks a feasible point (perhaps inexactly), considering
the original objective function and constraints.

In the optimality phase, it looks for a trial point that sufficiently reduces the value
of a Lagrangian defined by the original data in a tangent set that
approximates the feasible region, within a trust region centered at
the point obtained in the feasibility phase. Sufficient decrease of
a merit function which balances feasibility and optimality determines
the acceptance of the trial point obtained in the optimization
phase. If the trial point is not accepted, the size of the trust
region is reduced.

The purpose of this work is to offer an innovative alternative to solve DAP and a tangible application of the inexact restoration method.

This paper is organized as follows. In section 2 a model of the problem is presented as well as the assumptions made over the network. In section 3 the Inexact Restorarion method and its adaptation for bilevel problems are presented. In section 4 there is a complete description of the application of the Inexact Restoration to the DAP. A detailed presentation of every step of the algorithm is given and each subproblem is specifically treated in each subsection. Finally, in section 5 some numerical tests are presented. Conclusions are drawn in section 6.

\section{Model}

In this work, the DAP is considered as a mathematical program with equilibrium constraints (MPEC) where, for each demand, the flows are constrained to satisfy a deterministic  Wardrop's  user equilibrium (DUE) in the lower level, and an OD matrix is adjusted in the upper level taking into account a target matrix and some observed flows.

The transport network is represented as a directed graph $\mathcal{G}=(\mathcal{N},\mathcal{A})$ where $\mathcal{N}$ is the set of nodes and $\mathcal{A}$ is the set of directed links. $\mathcal{C}$ is chosen to represent the set of origin-destination pairs $(p,q)$.

Considering link flows, the DAP is formulated as:

\begin{center}(DAP)\hspace{2mm} $\min \hspace{1mm}F(v,d)=\eta_1
F_1(v)+\eta_2 F_2(d)$\\
\hspace{14mm} $s.a. \hspace{2mm} t(v)^t(v'-v)\geq 0, \; \forall
(v',d) \in \Omega,$
\end{center}
\vspace{2mm}

\noindent where $\Omega$ is the closed convex cone of pairs $(v,d)$ with $d\geq0$ and $v$ a feasible link flow for $d$, i.e. a
non-negative flow which satisfies the demand $d$. The equilibrium condition is expressed in terms of a variational inequality for the
associated link cost vector $t(v)$.
The function $F_1$ measures the deviation between the
assigned flow for the demand $d$ and the observed flow $\tilde v$,
in some links of the network $(\bar{\mathcal A}\subset \mathcal A)$.
The function $F_2$ measures the distance between $d$ and a target
matrix (usually an outdated OD-matrix $\tilde d$). The usually used
metrics are those of minimum squares, maximum entropy and maximum
likelihood (see \cite{C-F}). The parameters $\eta_1$ and $\eta_2$
reflect the confidence of the data $\tilde{g}$ and $\tilde{v}$
respectively.

For a general version of DAP, Chen and Florian proved in \cite{C-F},
under minor hypotheses of continuity of the functions $F_1,F_2$ and
$t$, that the problem admits at least one solution. In this work, the mapping $d\mapsto v^*(d)$, which assigns the equilibrium flows to a given demand $d$, is considered to be single
valued (i.e. the DUE admits an only one solution) and it is possible to write
$F(d)=F(v^*(d))$. For the problem to fit in this context it is necessary to make
some assumptions over the traffic network:

\begin{itemize}
\item the network is strongly connected, i.e. there exists at least
one route for each o-d pair;
\item the route cost functions are additive, i.e. they are the sum
of the link costs which constitute the route;
\item the link costs are separable, i.e. the flow in each link is
independent of the flow of all other links in the network;
\item the demand $d_{pq}$ is positive for each $(p,q)\in \mathcal{C}$
\item the link cost function $c_{a}:\mathbb{R}\mapsto\mathbb{R}$ is
positive, continuous and non decreasing for each $a \in
\mathcal{A}$.
\end{itemize}

These hypotheses guarantee the existence of equilibrium (both in the
link and route flow variables) and uniqueness of OD equilibrium
times. If each link cost function $c_{a}$ is
assumed to be strictly increasing, there is uniqueness of the equilibrium link flow
solution.

\section{Inexact Restoration and its adaptation to solve bilevel problems}

The Inexact Restoration Method (IRM) is motivated by the bad behavior of feasible methods in the presence of non linear constraints. To face these difficulties, the algorithms presented by Martinez et al. in \cite{Martinez98}, \cite{Martinez01} and \cite{MartinezPilotta00}, keep feasibility under control and are tolerant when the iterations are far from the solution. In \cite{MartinezPilotta05} there is an interesting overview of these algorithms and its main characteristics.

Originally, the IRM was designed to solve the problem
\begin{equation}\label{pbRI}\begin{array}{ll}{\min}&f(x)\\
s.t.&C(x)=0,\\
&x\in \Omega,
\end{array}
\end{equation}
\noindent where $f:\mathbb R^n\rightarrow \mathbb R$ and $C:\mathbb R^n\rightarrow\mathbb R^p$ are continuously differentiable functions and $S \subset \mathbb R^n$ is a closed convex set.

The algorithm consists of two well distinguished stages: feasibility
(or restoration stage) and optimality. It is an iterative method
which generates a sequence $x^k$ of feasible iterates with respect to
$\Omega$ but which not necessarily verifies $C(x)=0$. Precisely, the
restoration phase has the objective of moving the sequence in a
direction which generates a reduction of $||C(x)||$ and an auxiliary
sequence $y^k$, is built. In the second phase,  the optimality of
$y^k$ is improved by a minimization of a Lagrangian over a space
tangent to $\{C(x)=0\}$ in $y^k$.

The innovative use of the Lagrangian in the optimality phase has to
do with the fact that it behaves similarly both in the tangent space
and the feasible region. This may not be the case of the non linear
objective function.

The acceptance of the candidate $y^k$ depends on the value of a
merit function which combines feasibility and optimality.

\bigskip

Andreani et al. in \cite{Andreani} studied the possibility of adapting
the Inexact Restoration Method to solve bilevel problems. The
attractiveness of IRM had to do with the fact that this method may
allow solving these problems without reformulating them as single
level ones as most approaches for bilevel problems do. What is more,
the restoration phase gives the possibility of freely choosing a
method which improves feasibility. Consequently, if any globally
convergent algorithm is available to efficiently solve the lower
level problem, its structure could be exploited.

However, the adaptation to IRM for bilevel problems required further analysis.

\subsection{IRM adaptation for bilevel problems (IRMbi)}

Given a bilevel problem of the type
\begin{equation}\label{biRI}
\begin{array}{lll}
\min & F(x,y)\\
s.t. & x\in X\\
& y=\underset{y}{argmin}\; f(x,y)\\
&\begin{array}{ll}
s.t.&h(x,y)=0\\
 &y\geq 0
\end{array}
\end{array},
\end{equation}

\noindent to adapt the method which originally solves (\ref{pbRI}), the Karush Kuhn Tucker optimality conditions of the lower level problem are considered. In fact, they play the role of the constraint $C(x)=0$,
$$C(x,y,\mu,\gamma)=0, \;\; y\geq 0,\;\; \gamma \geq 0,$$
\noindent with
\begin{center}
$\begin{array}{lll}
C(x,y,\mu,\gamma)=\left(
\begin{array}{c}
\nabla_y f(x,y)+\nabla_y h(x,y)\mu-\gamma\\
h(x,y)\\
\gamma_1 y_1\\
\vdots\\
\gamma_m y_m
\end{array}
\right).
\end{array}$
\end{center}

The Lagrangian for the optimality phase
\begin{equation}\label{lag}
L(x,y,\mu,\gamma,\alpha)=F(x,y)+C(x,y,\mu,\gamma)^T \alpha.
\end{equation}

The restoration phase searches for a point
$z^k=(x^k,\bar y,\bar\mu,\bar\gamma)$ ``more feasible'' than the one built in the
previous iteration $s^k=(x^k,y^k,\mu^k,\gamma^k)$. To reach that
goal the lower level problem, parameterized in the variable $x^k$ is
solved. That is to say, a minimizer $\bar y$ and associated
multipliers $(\bar \mu,\bar\gamma)$ for the problem
\begin{equation}
\begin{array}{ll}
\underset{y}{\min}& f(x^k,y)\\
s.a.&h(x^k,y)=0\\
 &y\geq 0
\end{array}
\end{equation}
\noindent must be found. $z^k$ is defined as an intermediate point.
Then, a linear approximation, around $z^k$, of the feasible region
of the simplified problem (\ref{pbRIsimplif}) is built.
\begin{equation}\label{pbRIsimplif}
\begin{array}{lll}
\min & F(x,y)\\
& C(x,y,\mu,\gamma)=0\\
& s=(x,y,\mu,\gamma)\in \Omega\times\Delta
\end{array}
\end{equation}

\noindent where $\Omega\times\Delta$ represents the constraints
$x\in X, y\geq 0, \gamma\geq 0.$

The linear approximation in $z^k$ is the tangent space
$$\pi(z^k)=\{s\in \Omega\times\Delta : C'(z^k)(s-z^k)=0\}$$
and the Cauchy tangent direction $r_{tan}^k=r_{tan}(z^k)$ is
$$r_{tan}^k=P_k[z^k-\eta\nabla_s L(z^k,\alpha^k)]-z^k,$$
\noindent where $P_k [\cdot]$ is the orthogonal projection over the
space $\pi_k=\pi(z^k)$ and $L$ the Lagrangian presented above
(\ref{lag}). $r_{tan}^k$ is a feasible descent directon for
$L$ over $\pi_k$.

For the optimization phase a trust region centered in $z^k$ is
defined
$$\mathbb B_{k,i}=\{s\in \mathbb R^n : ||s-z^k||\leq \delta_{k,i}\},$$
\noindent and a candidate $v^{k,i}\in \mathbb B_{k,i}\cap \pi_k$ that reduces $L(\cdot,\alpha^k)$ is sought. The acceptance of $v^{k,i}$ depends on the value of a merit function. If it is rejected the trust radius is reduced and the scheme moves to an iteration $k,i+1$ until it finds the minimizer $z^{k,i^*}$.

The merit function used is:
$$\Psi(s,\alpha,\theta)=\theta L(s,\alpha)+(1-\theta)||C(s)||$$
\noindent where $\theta\in (0,1]$ is a penalty parameter that gives
different weights to the Lagrangian function and the feasibility.

With all these considerations the Inexact Restoration Method for
bilevel problems (IRMbi) was introduced. What is more, it was proved that there is global convergence to points which satisfy the Approximate Gradient Projection optimality conditions (AGP points).
For a detailed insight into these concepts see \cite{Andreani}.

\section{IRM for DAP}
\subsection{DAP as a bilevel problem}
The DAP was presented as a mathematical problem with equilibrium
constraints. However, Wardrop's user equilibrium can be obtained as
a solution to an optimization problem, the Traffic Assignment
Problem (TAP). The hypotheses under which this is true can be read in
\cite{TesisJor}.

In this case, DAP results in

\begin{equation}\label{DAP-TAP}
\begin{array}{l}
\min \;\;\;\; F(d,v)=\eta_1F_1(v)+\eta_2F_2(d)\\
\begin{array}{lll}
s.t. & \min  & T(v)=\underset{a\in \mathcal A}{\sum}\displaystyle\int_0^{v_a} c_a(s) ds\\
&s.t& \underset{r\in \mathcal{R}_{pq}}{\sum}
h_{pqr}= d_{pq}, \forall\; (p,q)\in \mathcal C, \\
& &h_{pqr}\geq 0, \forall \; r\in \mathcal{R}_{pq}, \forall\; (p,q)\in \mathcal C,\\
& &\underset{(p,q)\in\mathcal C}{\sum}\underset{r\in\mathcal
R_{pq}}{\sum} \delta_{pqra}h_{pqr}=v_a, \forall; a\in\mathcal A.\\
& d\geq 0.
\end{array}
\end{array}
\end{equation}

With this reformulation, DAP has bilevel structure. Consequently,
there exists the possibility of applying IRMbi to solve DAP. What is
more, for the lower level problem TAP, there exist globally
convergent methods to obtain the solution and in contrast to most of
the available methods for DAP, the complex structure of the traffic
assignment problem could be exploited.

\subsection{Change of variables for Karush Kuhn Tucker (KKT)  optimality conditions calculation}\label{cambvar}

It would be desirable to have KKT optimality conditions associated
to the lower level problem which are easy to handle. However, the
original version of TAP has a complex structure of the feasible set
due to the presence of two flow variables $v$ and $h$. To overcome this
difficulty, the TAP is reformulated in the node-arc version
presented in \cite{P}, as it is done in the non-heuristical approach
in \cite{LP:14}.

The new flow variable $X=(x_a^i)_{a\in\mathcal A,i\in \mathcal C}$
represents the arc flow disaggregated by demand. $X\in\mathbb
R^{|\mathcal A||\mathcal C|}$ is a column vector.

In this context, Wardrop's user equilibrium condition is rewritten
as
$$T(X^*)^T(X-X^*)\geq 0, \forall\; X\in \tilde\Omega(d)$$
\noindent where $\tilde\Omega(d)=\{X\geq 0 : \Gamma d-M X=0\}$.

The function $T$ and the matrices $\Gamma$ and $M$ verify:
$$T(X)=R^T t(RX)$$
with $R\in\mathbb R^{|\mathcal A|\times |\mathcal A||\mathcal C|}$
defined as
$$R=(\underbrace{I_{|\mathcal A|},...,I_{|\mathcal A|}}_\text{$|\mathcal C|$ times}),$$
\noindent and $I_{|\mathcal A|}$ the identity matrix in $\mathbb
R^{|\mathcal A|\times|\mathcal A|}.$

$$\Gamma=\begin{pmatrix}
\gamma^1 & 0 & \cdots & 0\\
0 & \gamma^2 & \ddots & \vdots\\
\vdots & \ddots & \ddots & 0\\
0 & \cdots & 0 &\gamma^{|\mathcal C|}
 \end{pmatrix}\in \mathbb R^{|\mathcal C||\mathcal N|\times |\mathcal C|},$$
\noindent with $\gamma^i=(\gamma_k^i)_{k\in\mathcal N}$ such that
$\gamma_k^i=\left\{\begin{array}{ll}
-1 & \text{if $k$ is the origin node for the demand $i$},\\
1 &  \text{if $k$ is the destination node for the demand $i$},\\
0 & \text{otherwise.}
\end{array}
\right.$
$$M=\begin{pmatrix}
A & 0 & \cdots & 0\\
0 & \ddots & \ddots & \vdots\\
\vdots & \ddots & \ddots & 0\\
0 & \cdots & 0 & A
 \end{pmatrix}\in \mathbb R^{|\mathcal C||\mathcal N|\times |\mathcal C||\mathcal A|},$$
\noindent with $A\in \mathbb R^{|\mathcal N|\times|\mathcal A|}$ being
the node-arc incidence matrix.

Finally, the KKT system for this reformulation of the lower level
problem results in:

$$\left\{ \begin{array}{rll}
T(X)+M^T\alpha-\beta&=&0,\\
\Gamma d-MX&=&0,\\
\beta^T X&=&0,\\
\beta\geq 0, X\geq 0,
\end{array}\right.$$

Here, $\alpha$ is the multiplier vector associated to the equality
constraints and $\beta$ the multiplier vector associated to
the inequality constraints.

\subsection{IRMbi for DAP}

Having done the change of variables presented above (Section
\ref{cambvar}), the goal of applying IRMbi to solve DAP is
re established.

Firstly, the simplified problem (\ref{pbRIsimplif}) for DAP,
together with the KKT system obtained, is written:
\begin{subequations}\label{DAPkkt}\begin{align}
\min &\;F(d,X)=\eta_1 F_1(RX)+\eta_2
F_2(d)\\
s.a. &\;\; T(X)+M^T\alpha-\beta=0\label{r1}\\
&\;\;\Gamma d-MX=0\label{r2}\\
&\;\;\beta^T X=0\label{r3}\\
&\;\;\beta\geq 0, X\geq 0, d\geq 0\label{r4}
\end{align}
\end{subequations}

Then, choosing
$$C(d,X,\alpha,\beta)=\left(\begin{array}{c}
T(X)+M^T\alpha-\beta\\
\Gamma d-MX\\
\beta^T X\\
\end{array}\right)\; and$$

$\Omega\times\Delta=\{s=(d,X,\alpha,\beta)^T\in\mathbb
R^{|\mathcal C|}\times\mathbb R^{|\mathcal C||\mathcal
A|}\times\mathbb R^{|\mathcal C||\mathcal N|}\times\mathbb
R^{|\mathcal C||\mathcal A|}:\; d\geq 0 \wedge X\geq 0 \wedge
\beta\geq 0\}$,  it results in:
\begin{equation}\label{DAPrisimplif}\begin{array}{lc} \min &\;
F(d,X)=\eta_1 F_1(RX)+\eta_2
F_2(d)\\
s.a. & C(s)=0,\\
&s\in \Omega\times\Delta.\\
\end{array}
\end{equation}

Having checked that the restoration phase can be carried out for DAP
under some considerations, the same must be done for the
optimization phase.

A linear approximation of the feasible set defined by the
constraints of (\ref{DAPrisimplif}) must be considered. For a point
$z=(g^*,X^*,\alpha^*\beta^*)$ it results:

$$\pi_z=\{s\in \Omega\times \Delta: C'(z)(s-z)=0\}$$

\noindent where

$$C'(z)=\begin{pmatrix}0&T'(X^*)& M^T&-I_{|\mathcal A||\mathcal C|}\\
\Gamma &-M & 0 & 0\\
0&I_{\beta^*}&0&I_{X^*}
\end{pmatrix},$$

\noindent $I_{|\mathcal A||\mathcal C|}$ is the identity matrix of
dimensions $|\mathcal A||\mathcal C|$,

\noindent $I_{\beta^*}$ is a matrix of zeros which in its
diagonal has the entries of vector $\beta^*$ and

\noindent $I_{X^*}$ is a matrix of zeros which in its diagonal has
the entries of vector $X^*$.

The matrix $C'(z)$, $C'_z$ for simplicity, always exists and can be
easily obtained. Consequently, it is possible to obtain the tangent
space $\pi_z$. $C'_z$ is a fixed matrix throughout the iterations,
and the linearization $\pi_z$ is the set of $s\in \Omega\times
\Delta$ that are solutions to the linear system: $C'_z(s-z)=0.$

\subsection{Algorithm}

In this section the complete scheme adapted for DAP is presented. The details of implementation are  given in section 4.5.

\bigskip

The following constants are fixed. $\eta>0$, $M>0$,
$\theta_{-1}\in(0,1),$ $ \delta_{min}>0,\tau_1>0, \tau_2>0$. $k=0$.

Let $s^0=(d^0,X^0,\alpha^0,\beta^0)\in \mathbb R^{|\mathcal
C|}\times \mathbb R^{|\mathcal C||\mathcal A|}\times \mathbb
R^{|\mathcal C||\mathcal N|}\times \mathbb R^{|\mathcal C||\mathcal
A|}$ be an initial approximation, $\mu^0$ an initial approximation
of the multiplier and $\omega^i$ a sequence of positive numbers such
that: $\sum\limits_{i=0}^{\infty}\omega^i<\infty$.

\bigskip

\textcolor{blue}{Step 1. Penalty parameter initialization.}

$\theta_k^{min}=min\{1,\theta_{k-1},...,\theta_{-1}\}$,

$\theta_k^{large}=min\{1,\theta_k^{min}+\omega^k\}$,

$\theta_{k,-1}=\theta_k^{large}$.

\bigskip

\textcolor{blue}{Step 2. Restoration Phase.}

Solve the traffic assignment problem for $d=d^k$ and get the
Lagrangian multipliers associated to the obtained equilibrium.

Let $X^*$ be the equilibrium solution and $\alpha^*, \beta^*$
the associated multipliers.

Define $z^k=(d^k,X^*,\alpha^*,\beta^*)$.

\bigskip

\textcolor{blue}{Step 3. Cauchy tangent direction.}

Calculate $r_{tan}^k=P_k[z^k-\eta\nabla_s L(z^k,\mu^k)]-z^k$.

$\pi_k=\{s\in\Omega\times \Delta : C'(z^k)(s-z^k)=0\}$

* If $z^k=s^k$ and $r_{tan}^k=0$, finish. $(d^k,X^k)$ is the solution to DAP.

* Otherwise, $i=0$, $\delta_{k,0}\geq\delta_{min}$ and move to Step 4.

\bigskip

\textcolor{blue}{Step 4. Optimization Phase in $\pi_k$.}

* If $r_{tan}^k=0$, set $v^{k,0}=z^k$.

* Otherwise, calculate $t_{break}^{k,i}=min\{1,\delta_{k,i}/
||r_{tan}^k||\}$ and get $v^{k,i}$ such that:

\hspace{5mm}- $v_{k,i}\in\pi_k$,

\hspace{5mm}- $||v^{k,i}-z^k||_{\infty}<\delta_{k,i}$,

\hspace{5mm}- for some $t\in (0, t_{break}^{k,i}]$,

$L(v^{k,i}, \mu^k)\leq max\{L(z^k+tr_{tan}^k,\mu^k), L(z^k, \mu^k)-
\tau_1\delta_{k,i}, L(z^k,\mu^k)-\tau_2\}$.

\bigskip

\textcolor{blue}{Step 5. Trial multipliers.}

* If $r_{tan}^k=0$ set $\mu_{trial}^{k,i}=\mu^k$.

* Otherwise, calculate $\mu_{trial}^{k,i}\in\mathbb
R ^{2|\mathcal C||\mathcal A|+|\mathcal C||\mathcal N|}$ such that
$|\mu_{trial}^{k,i}|\leq M$.

\bigskip

\textcolor{blue}{Step 6. Predicted reduction.}

Define $\forall\; \theta\in [0,1]$,

$Pred_{k,i}(\theta)=\theta[L(s^k,\mu^k)-L(v^{k,i},\mu^k)-C(z^k)^T(\mu_{trial}^{k,i}-\mu^k)]+(1-\theta)[|C(s^k)|-|C(z^k)|].$

\bigskip

Compute $\theta_{k,i}$ as the maximum $\theta\in[0, \theta_{k,i-1}]$
which verifies
$$Pred_{k,i}(\theta)\geq \dfrac{1}{2}[|C(s^k)|-|C(z^k)|].$$

Define $Pred_{k,i}=Pred_{k,i}(\theta_{k,i})$.

\bigskip

\textcolor{blue}{Step 7. Compare actual and predicted reduction.}

Calculate
$Ared_{k,i}=\theta_{k,i}[L(s^k,\mu^k)-L(v^{k,i},\mu_{trial}^{k,i})]+(1-\theta_{k,i})[|C(s^k)|-|C(v^{k,i}|)].$

* If $Ared_{k,i}\geq 0.1 Pred_{k,i}$ UPDATE:

$s^{k+1}=v^{k,i}$, $\mu^{k+1}=\mu_{trial}^k$,
$\theta_k=\theta_{k,i}$, $\delta_k=\delta_{k,i}$, $k=k+1$,

and TERMINATE iteration $k$.

* Otherwise, choose

\hspace{5mm} - $\delta_{k,i+1}\in [0.1\delta_{k,i}, 0.9
\delta_{k,i}]$,

\hspace{5mm} - i=i+1,

and move to step 4.

\subsection{Implementation issues}

So far in this work, both phases of the algorithm IRMbi have been
revised and analyzed for the DAP. In this section, details of
implementation for each step of the scheme are given.

\subsubsection{Solving the traffic assignment problem: Step 2}

For the bilevel problem DAP, there exist algorithms which efficiently
solve the lower level problem: TAP. In this work, the Disaggregated Simplicial Decomposition (DSD) algorithm is chosen.
Particularly, the version implemented in CiudadSim (Scilab Toolbox \cite{CS}, \cite{manual}) is used, as
it gives the possibility of working with the flow variable
disaggregated by demand. Precisely, this variable is an auxiliary
one and it is available without any modifications to the code of the
DSD, except for the output.

Even though the DSD algorithm solves TAP for a fixed demand $d^k$
providing a solution $X^k$, the associated multipliers $\alpha^k$
and $\beta^k$ must be obtained to build the intermediate point
$z^k=(d^k,X^k,\alpha^k,\beta^k)$, and this cannot be done
through the DSD.

However, the KKT system associated to TAP  always admits solutions
$\alpha^k$ and $\beta^k$. That is to say, for a given demand $d^k$
and the associated equilibrium vector $X^k$, there exist $\alpha^k$
and $\beta^k$ which satisfy the system:

$$\left\{ \begin{array}{rll}
T(X^k)+M^T\alpha-\beta&=&0,\\
\Gamma d^k-MX^k&=&0,\\
\beta^T X^k&=&0,\\
\beta\geq 0, X^k\geq 0,
\end{array}\right.$$

The above assertion is possible due to the linearity of the problem's constraints and the fact that there is a solution existence proof for TAP.

\noindent

To obtain a pair $(\alpha, \beta)$ compatible with
$(d^k,X^k)$, the following system is solved:
\begin{subequations}\label{MulTAP0}\begin{align}
T(X^*)+M^T\alpha-\beta&=0,\label{MulTAP1}\\
\beta^T X^*&=0,\\
\beta\geq 0,
\end{align}
\end{subequations}

The subroutine ``linsolve'' from ScicosLab 4.3 is used to solve
(\ref{MulTAP0}). This algorithm solves optimization problems with
linear constraints and consequently will provide a solution which is
a feasible point, that is to say, a pair $(\alpha, \bar \alpha)$
that satisfies (\ref{MulTAP0}) as it is needed.

\subsubsection{Building the Cauchy tangent direction: Step 3}

In this step a projection problem must be solved. Precisely, the
projection of a vector $z-v$ over the tangent space $\pi_z$ is
needed. To calculate it, the following optimization problem is
solved:
\begin{equation}\label{pbproy}
\begin{array}{ll}
\underset{s}{\min} & \dfrac{1}{2}||z-v-s||^2\\
s.a. &C'_z(s-z)=0.
\end{array}
\end{equation}

Here, the vector $v=-\eta\nabla_s L(z^k,\mu^k)$.

The optimality conditions of the problem are studied. Under
appropriate hypotheses which state non singularity of the matrix
$C'_z$ (see Lemma 4 in \cite{TesisJor}), the existence of the Cauchy
tangent direction is proved.

The subroutine ``quapro'' from ScicosLab 4.3 which solves quadratic
problems with linear constraints is chosen to solve (\ref{pbproy})
numerically.

Step 3 also includes the stopping condition. This is satisfied
by any AGP point. See \cite{Andreani} for more details.

To check the stopping condition for the candidate $z^k$ numerically,
the following test is carried out: if
\begin{center}
$||z^k-s^k||<\varepsilon_1$ y $||r_{tan}^k||<\varepsilon_2,$
\end{center}
\noindent for $\varepsilon_1$ and $\varepsilon_2$ small, the
algorithm is stopped.

\subsubsection{Finding the candidate $v^{k,i}$ which improves
 optimality: Step 4}

 The original version of IRMbi gives freedom to choose the method to
 find $v^{k,i}$ which satisfies all the conditions stated.

 The optimization phase is carried out with the objective of
 making a descent of the value of the Lagrangian. The Cauchy tangent
 direction is always a descent direction for such Lagrangian as it
 is proved in \cite{Martinez98}. However, in \cite{Martinez01} it is
 stated that $v^{k,i}=z^k+t r_{tan}^k$ may not always be the best
 candidate.

 $v^{k,i}$ must satisfy simultaneously:
 \begin{itemize}
\item $v^{k,i}\in\pi_k$,
\item $||v^{k,i}-z^k||_{\infty}<\delta_{k,i}$,
\item for some $t\in (0, t_{break}^{k,i}]$,
$$L(v^{k,i},\mu^k)\leq max\{L(z^k+tr_{tan}^k,\mu^k),
L(z^k,\mu^k)- \tau_1\delta_{k,i}, L(z^k,\mu^k)-\tau_2\}.$$
\end{itemize}

The last one is a descent condition, see Figure \ref{condMax}:

\begin{figure}[h!]
\centering
\includegraphics[width=5cm]{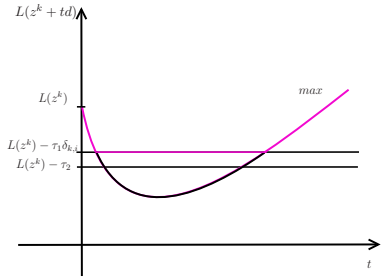}
\caption{Descent condition over the candidate $v^{k,i}$}
\label{condMax}
\end{figure}

To find such $v^{k,i}$ an algorithm proposed by Martinez in
\cite{Martinez98} is used.

The following auxiliary problem is considered:
\begin{equation}\label{pbaux}\begin{array}{ll}\min & L(v, \mu^k)\\
s.a.& C'_{z^{k}}(v-z^{k})=0,\\
& ||v-z^k||_{\infty}\leq \delta_{k,i}.
\end{array}
\end{equation}

The solution to this linearly constrained problem is
undoubtedly a candidate for $v^{k,i}$. However, it is not be
necessary to solve the problem to find an appropriate $v^{k,i}$. The
successive iterates generated by the algorithm which solves
(\ref{pbaux}) are tested and the scheme is stopped as soon as there
is one approximation which verifies all the conditions for $v^{k,i}$.

Each iteration of the algorithm is associated to a fixed $v$, and a
search direction $r_v$ is calculated. Precisely, $r_v=P_S(v-\nabla
F(v))-v$, where $S$ is the feasible region of problem (\ref{pbaux}).
Then, a backward linear search is carried out until the norm of the
direction is less than $10^{-3}$ or $100$ points have been tested.

To solve the projection problem, the associated minimum problem is
considered:
\begin{equation}\label{Proypaso4}\begin{array}{ll}\min&\dfrac{1}{2}||v-\nabla F(v) - w||^2\\
s.a. & ||w-z^k||_{\infty}\leq \delta_{k,i},\\
& w\in\pi_k.
\end{array}
\end{equation}

This problem has a solution due to the fact that it consists in the
problem of  minimizing a continuous function over a compact set.
What is more, the solution $w^*$ allows building in each iteration
of the mentioned algorithm the direction $r_v=w^*-v$, a feasible
and descent direction for $L(v,\mu^k)$.

The above assertion is proved in the following lemma:
\begin{lem}
The direction $r_{v}=P_S(v-\nabla F(v))-v,$ where $S$ is the feasible set which the
constraints in (\ref{Proypaso4}) describe, is a feasible direction.
What is more, $r_v$ is a descent direction for $L(v,\mu^k)$.
\end{lem}
\begin{proof}
To see that $r_v$ is a feasible direction it is checked that
there exists $\varepsilon > 0$ such that $v+\alpha r_v \in S
\;\forall \alpha\in [0,\varepsilon)$. In fact, $v\in S$ due to the
fact that it is an approximation built by the proposed scheme. Let
$u=v+\alpha r_v=v+\alpha (w^*-v)$ where $w^*$ is a solution to
(\ref{Proypaso4}) and consequently verifies $w^*\in S$. Re-writing,
$u=(1-\alpha)v+\alpha w^*$ with $S$ convex, it results in $u\in S$ if
$\alpha\in[0,1]$.

To see that $r_v$ is a descent direction for $L(v,\mu^k)$, it is
first proved that it is a descent direction for $F(v)$. $r_v\neq 0$ is
assumed. Then, $w^*\neq v$ and due to the fact that $w^*\in S$, it
results in: $||w^*-(v-\nabla F(v))||_2^2<||v-(v-\nabla F(v))||_2^2$,
then,

$||w^*-v||_2^2+2\langle w^*-v, \nabla F(v)\rangle+||\nabla
F(v)||_2^2<||\nabla F(v)||_2^2$, and consequently:
$$\langle r_v, \nabla F(v)\rangle<0.$$

Taking into account that $r_v$ belongs to  $Ker (C'_{z^k})$, in
fact,
$$C'_{z^k}d_v=C'_{z^k}(w^*-v)=C'_{z^k}(w^*-z^k+z^k-v)=0,$$
considering that $w^*$ and $v$ are both in $\pi_k$, it results in:
$$\langle r_v, \nabla L(v,\mu)\rangle=\langle r_v, \nabla F(v)\rangle<0$$
\noindent as it was desired to prove. \qed
\end{proof}

Numerically, the descent direction is obtained by solving problem
(\ref{Proypaso4}) with the subroutine ``quapro''  from ScicosLab 4.3.
A maximum of $10$ iterations are performed and each approximation is
tested as a possible candidate $v^{k,i}.$

\section{Numerical experiments: Validation test example}

A toy problem has been chosen as an expository device to illustrate
the applicability of the Inexact Restoration Method for bilevel
problems to solve DAP.

The transport network has 3 nodes, 4 links and 2 demands represented
by the pink arrows.

\begin{figure}[h!]
\centering
\includegraphics{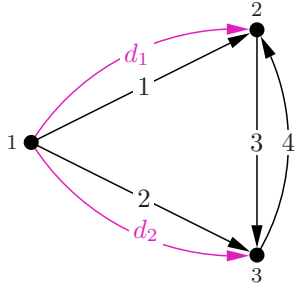}
\caption{Validation test example} \label{ejemploRI}
\end{figure}

The link flow variable, dissagregated by demand, is in this case:

$$X=(x_1^1\;x_2^1\;x_3^1\;x_4^1\;x_1^2\;x_2^2\;x_3^2\;x_4^2)^T,$$
where $x_i^j:$ represents the flow in arc $i$ associated to the
demand $j$ and consequently $x_i=x_i^1+x_i^2$. X defined in this way
verifies $X\in\mathbb R^8.$

The matrices $R\in \mathbb R^{4\times 8}$, $\Gamma \in \mathbb
R^{6\times 2}$ and $M\in \mathbb R^{6\times 8}$ are in this case:

$$R=\begin{pmatrix}1&0&0&0&1&0&0&0\\
0&1&0&0&0&1&0&0\\
0&0&1&0&0&0&1&0\\
0&0&0&1&0&0&0&1
\end{pmatrix},$$
$$\Gamma=\begin{pmatrix} - 1&0\\
1&0\\
0&0\\
0&-1\\
0&0\\
0&1
\end{pmatrix},\;
M=\begin{pmatrix} - 1&- 1&0&0&0&0&0&0\\
1&0&- 1&1&0&0&0&0\\
0&1&1&- 1&0&0&0&0\\
0&0&0&0&- 1&- 1&0&0\\
0&0&0&0&1&0&- 1&1\\
0&0&0&0&0&1&1 &- 1
\end{pmatrix}\;$$

and

$$\begin{array}{lll}T(X)&=&R^Tt(RX)\\
&=&(t_1(x_1)\; t_2(x_2)\; t_3(x_3)\; t_4(x_4)\;
t_1(x_1)\; t_2(x_2)\; t_3(x_3)\; t_4(x_4))^T\\
&=&(x_1\; x_2\; x_3\; x_4\; x_1\; x_2\; x_3\; x_4)^T.\end{array}$$

The numerical tests are carried out considering a known target
demand and observed flows in arcs 1 and 2 which correspond to an
affectation of such demand. The purpose of this approach is to
guarantee that there exists a global minimum where the objective
function of the associated DAP assumes value zero.

The constants were fixed as follows: $d_1=1.5$, $d_2=1.75$, $\tilde v_1=1.5833333$, $\tilde v_2=1.6666667$, $\eta_1=0.5$,
$\eta_2=0.5$, $F_1=||v-\tilde v||^2$, $F_2=||d-\overline d||^2$.

In the following table we present the details of the experiments
carried out and the results obtained:

\begin{table}[!h]
\centering

        \begin{tabular}{||c|c|c|c|c||}\hline\hline
            \textbf{Exp.} & Initial demand $d^{0_i}$& N$^\circ$ it & Obj Value & Dem \\ \hline
            1 & $\begin{pmatrix}1&2 \end{pmatrix}$ & $14$ & $0.003475$ & $(1.498189 \; 1.751238)$\\
            2 & $\begin{pmatrix}1&1 \end{pmatrix}$ & $10$ & $0.020833$ & $(1.624989 \;1.624989)$\\
            3 & $\begin{pmatrix}1&1.5 \end{pmatrix}$ & $12$ & $0.003475$ & $(1.498139 \;1.751272)$\\
            4 & $\begin{pmatrix}1.8&2 \end{pmatrix}$ & $9$  & $0.003472$ & $(1.499572 \; 1.750095)$\\ \hline \hline
        \end{tabular}
        \caption{Experiments details}\label{Exp}

\end{table}

\noindent The initial value for the variable $s$ for each experiment
is
$$s^{0_i}=(d^{0_i},(0\;0\;0\;0\;0\;0\;0\;0),(0\;0\;0\;0\;0\;0),(0\;0\;0\;0\;0\;0\;0\;0))^T,
i=1,2,3,4.$$

\subsection{Comments}
For three of a total of four experiments (Exp. 1, 3 and 4
precisely), convergence to a global optimum of the problem was
registered. However, this was not the case for experiment 2. The
iterations got stuck around a point which is not a global minimum
of the problem but which verifies the AGP optimality condition. What
is more, the objective function assumes over such point a value
which is close to the minimum value of the problem.

\section{Conclusions}
In this work an application of the Inexact Restoration method for bilevel problems to a real problem, the DAP, has been presented. The advantages of the method have been exploited. Few of the available methods to treat DAP maintain the structure of the lower level problem TAP as IRMbi does. Most methods deal with the single level version of the DAP. What is more, for IRMbi there are proofs of convergence to AGP points while others are just heuristics or descent methods.

In the feasibility phase the TAP was solved exactly through available software. In the optimality phase a descent method for the Lagrangian proposed by Martinez was implemented.

Some numerical tests over a small network were carried out and convergence to global optimum was obtained in 3 out of 4 cases. In the remaining case, convergence to an AGP point was achieved.

When applied to real size networks, this formulation leads to very large-scale problems. Consequently, future research will be directed towards avoiding the dissagregated flow variable in order to obtain computationally treatable problems.

\section*{Acknowledgements}
This work was partially supported by CONICET (PIP 2012-2014 N$^\circ$ 0286), FONCyT (Pict 2012-2212), SPU (3325/15c Proy.31-65-128) AND Universidad Nacional de Rosario (ING 428),  Argentina.

%%=========================================


\begin{thebibliography}{1}
%%=========================================

\bibitem{Andreani}
R.~Andreani, S.~L.~C. Castro, J.L. Chela, A.~Friedlander, and S.~A. Santos.
\newblock An inexact-restoration method for nonlinear bilevel programming
  problems.
\newblock {\em Computational Optimization and Applications}, 43:307--328, 2009.

\bibitem{C-F}
Y.~Chen and M.~Florian.
\newblock O-d demand adjustment problem with congestion: Part i. model analysis
  and optimality conditions.
\newblock {\em Advanced Methods in Transportation Analysis}, pages 1--22, 1996.

\bibitem{CS}
CiudadSim:.
\newblock http://www-rocq.inria.fr/metalau/ciudadsim/.

\bibitem{C-B}
E.~Codina and J.~Barcel\'o.
\newblock Adjustment of {O}-{D} trip matrices from observed volumes: An
  algorithmic approach based on conjugate directions.
\newblock {\em European Journal of Operational Research}, 155:535--557, 2004.

\bibitem{C-M}
E.~Codina and L.~Montero.
\newblock Approximation of the steepest descent direction for the o-d matrix
  adjustment problem.
\newblock {\em Annals Operational Research}, 144:329--362, 2006.

\bibitem{IPS:12}
A.~F. Izmailov, A.~L. Pogosyan, and M.~V. Solodov.
\newblock Semismooth {N}ewton method for the lifted reformulation of
  mathematical programs with complementarity constraints.
\newblock {\em Computational Optimization and Applications}, 51(1):199--221,
  2012.


\bibitem{manual}
P.~A. Lotito, E.~M. Mancinelli, J.~P. Quadrat, and L.~Wynter.
\newblock The traffic assignment toolboxes of scilab.
\newblock {\em INRIA - Rocquencourt}, 2003.

\bibitem{LP:14}
P.~A. Lotito and L.~A. Parente.
\newblock A non heuristical approach for the bilevel {O-D} matrix estimation
  problem from traffic counts.
\newblock {\em Preprint}, 2015.

\bibitem{L-P}
J.~T. Lundgren and A.~Peterson.
\newblock A heuristic for the bilevel origin-destination-matrix estimation
  problem.
\newblock {\em Transport Research B}, 42:339--354, 2008.

\bibitem{Martinez98}
J.~M. Martinez.
\newblock Two-phase model algorithm with global convergence for nonlinear
  programming.
\newblock {\em Journal of Optimization Theory and Applications},
  96(2):397--436, 1998.

\bibitem{Martinez01}
J.~M. Martinez.
\newblock Inexact-restoration method with Lagrangian tangent decrease and new
  merit function for nonlinear programming.
\newblock {\em Journal of Optimization Theory and Applications}, 111(1):39--58,
  2001.

\bibitem{MartinezPilotta00}
J.~M. Martinez and E.~A. Pilotta.
\newblock Inexact-restoration algorithm for constrained optimization.
\newblock {\em Journal of Optimization Theory and Applications}, 104(1):135--163,
  2000.

\bibitem{MartinezPilotta05}
J.~M. Martinez and E.~A. Pilotta.
\newblock Inexact-restoration methods for nonlinear programming: Advances and perspectives.
\newblock In L.~Qi, K. Teo and X. Yang editors, {\em Optimization and Control with Applications}, pages 271--292, Springer,
  2005.

\bibitem{Nguyen}
S.~Nguyen.
\newblock Estimating an o-d matrix from network data: A network equilibrium
  approach.
\newblock {\em Publication 87, Centre de recherche sur les transports (CRT),
  Universit\'e de Montr\'eal, Montr\'eal, Canada}, 1977.

\bibitem{P}
M.~Patriksson.
\newblock {\em The Traffic Assignment Problem. Models and Methods}.
\newblock Topics in Transportation VSP, The Netherlands, 1994.

\bibitem{Spiess90}
H.~Spiess.
\newblock A descent based approach for the od matrix adjustment problem.
\newblock {\em Publication 693, Centre de recherche sur les transports (CRT),
  Universit\'e de Montr\'eal, Montr\'eal, Canada.}, 1990.

\bibitem{Ste:12}
O.~Stein.
\newblock Lifting mathematical programs with complementarity constraints.
\newblock {\em Mathematical Programming}, 131:71--94, 2012.

\bibitem{TesisJor}
J.~Walpen.
\newblock Sobre la resoluci\'on del problema de ajustar la matriz origen
  destino en una red de tr\'afico vehicular congestionada.
\newblock {\em Tesis de Doctorado en Matem\'atica - FCEIA - Universidad
  Nacional de Rosario,Director: Pablo A. Lotito. Codirectora: Elina M.
  Mancinelli}, 2015.

\bibitem{WML}
J.~Walpen, E.~M. Mancinelli, and P.~A. Lotito.
\newblock A heuristic for the od matrix adjustment problem in a congested
  transport network.
\newblock {\em European Journal of Operational Research}, 242:807--819, 2015.

\bibitem{Yangetal}
H.~Yang, T.~Sasaki, Y.~Iida, and Y.~Asakura.
\newblock Estimation of origin-destination matrices from link traffic counts on
  congested networks.
\newblock {\em Transportation Research}, 26B:417--434, 1992.


\end{thebibliography}
\end{document}